\documentclass[11pt]{amsart}
\usepackage{amssymb}
\usepackage{ bbold }
\usepackage{xypic}
\xyoption{all}
\usepackage{url}
\usepackage{mathrsfs}
\usepackage{tikz-cd}
\usepackage{eucal}

\title[A family of infinite degree tt-rings]{A family of infinite degree tt-rings}

\author[Juan Omar G\'omez]{Juan Omar G\'omez}
\thanks{}
\address{
\hfill\break Centro de Investigaci\'on en Matem\'aticas, A.C., Unidad M\'erida \\
\hfill\break Parque Cient\'ifico y Tecnol\'ogico de Yucat\'an  \\ 
\hfill\break Carretera Sierra Papacal--Chuburn\'a Puerto Km 5.5 \\
\hfill\break Sierra Papacal, M\'erida, YUC 97302 \\
\hfill\break Mexico.}
\email{juan.gomez@cimat.mx}

\newcommand{\comments}[1]{}

\newcommand{\StMod}{\operatorname{StMod}\nolimits}

\def \K{{\mathcal K}}

\def \P{{\mathcal P}}

\makeatletter
\newcommand*{\doublerightarrow}[2]{\mathrel{
  \settowidth{\@tempdima}{$\scriptstyle#1$}
  \settowidth{\@tempdimb}{$\scriptstyle#2$}
  \ifdim\@tempdimb>\@tempdima \@tempdima=\@tempdimb\fi
  \mathop{\vcenter{
    \offinterlineskip\ialign{\hbox to\dimexpr\@tempdima+1em{##}\cr
    \rightarrowfill\cr\noalign{\kern.5ex}
    \rightarrowfill\cr}}}\limits^{\!#1}_{\!#2}}}
\newcommand*{\triplerightarrow}[1]{\mathrel{
  \settowidth{\@tempdima}{$\scriptstyle#1$}
  \mathop{\vcenter{
    \offinterlineskip\ialign{\hbox to\dimexpr\@tempdima+1em{##}\cr
    \rightarrowfill\cr\noalign{\kern.5ex}
    \rightarrowfill\cr\noalign{\kern.5ex}
    \rightarrowfill\cr}}}\limits^{\!#1}}}
\makeatother

\theoremstyle{plain}

\newtheorem{theorem}{Theorem}[section]

\theoremstyle{definition}
\newtheorem{definition}[theorem]{Definition}
\newtheorem{remark}[theorem]{Remark}
\newtheorem{example}[theorem]{Example}

\thanks{The author acknowledges support by CONACYT under the program ``Becas Nacionales de Posgrado''.}
\date{\today}

\begin{document}

\maketitle

\begin{abstract}
    We construct a family of infinite degree tt-rings, giving a negative answer to an open question by P. Balmer. 
\end{abstract}

\section*{Introduction}

The abstraction of tensor-triangular geometry makes it possible to connect ideas and unify techniques in the study of tensor triangulated categories arising in different areas of mathematics, including algebraic geometry, commutative algebra, modular representation theory and stable homotopy theory. We refer to \cite{Bal10} for an account of \textit{standard} tensor triangulated categories.  

A \textit{tt-ring} is a separable commutative algebra object in a tensor triangulated category. These objects play an important role in tensor-triangular geometry. For instance, the Eilenberg-Moore category of modules over a tt-ring remains a tensor triangulated category, and extension of scalars is a tt-functor. The notion of degree introduced by Balmer in \cite{Bal11} has been successfully exploited in many applications: notably, in establishing a connection between the Going-Up Theorem and Quillen's Stratification Theorem (see \cite{Bal16}) and generalizing the  \'etale topology (see \cite{NP23}), both in the setting of tensor-triangular geometry.  

Remarkably, all tt-rings in standard tensor triangulated categories have finite degree \cite[Section 4]{Bal11}. It is an open question in \cite[Page 2]{Bal11} whether the degree of a tt-ring must always be finite. We provide a family of infinite degree tt-rings, giving a negative answer to this question. In fact, this family extends to a family of infinite degree rigid-compact tt-rings in the framework of rigidly-compactly generated tensor triangulated categories.

\section{Infinite degree tt-rings}

For $i\in \mathbb{N}$, let $\K_i$ be a non-trivial essentially small tensor triangulated category. Define $$\K:=\prod_{i\in\mathbb{N}} \K_i.$$  It is clear that $\K$ is essentially small; the product of small skeletons in each component defines a small skeleton of $\K$. We give $\K$ a triangulated structure and a symetric monoidal structure, both component-wise. In particular, $\K$ is a non-trivial essentially small tensor triangulated category.

\begin{theorem}\label{example of infinite degree}
 Let $\K_n$ and $\K$ as above and let $\mathbb{1}_n$ denote the monoidal unit of $\K_n$. Then the tt-ring $$A:=(\mathbb{1}_n^{\times n})_{n\in \mathbb{N}}\in \K$$ has infinite degree with the component-wise tt-ring structure.  
\end{theorem}

\begin{proof}
    It is clear that $A$ is a tt-ring with component-wise multiplication, and a component-wise bilinear section. On the other hand, by the definition of $\K$, the projection functor $$\mathrm{pr}_n\colon \K\to\K_n$$ is a tensor triangulated functor for each $n\geq1$. In particular, $\mathrm{pr}_n(A)= \mathbb{1}_n^{\times n}$ which has finite degree $n$ (see \cite[Theorem 3.9]{Bal11}). Then $A$ has infinite degree, otherwise it contradicts \cite[Theorem 3.7]{Bal11}. 
\end{proof}

\begin{remark}
By \cite[Theorem 3.8]{Bal11}, it follows that there exists a prime $\P$ in $\K$ such that the tt-ring $ q_\P (A) $ has infinite degree in  $\K_\P$.  Therefore placing the adjective \textit{local}  on an essentially small tensor triangulated category is not enough to guarantee that tt-rings have finite degree.
\end{remark}

At first glance, our example of a tt-ring of infinite degree seems to live in an artificial tensor triangulated category. However, it is possible to find this type of example in practice, for instance in the study of stable module categories for infinite groups. 

Let $\mathrm{CAlg}(\mathrm{Pr}^L_{\textrm{st}})$ denote the $\infty$-category of stable homotopy theories, that is, presentable, symmetric monoidal, stable $\infty$-categories with cocontinuous tensor product in each variable\footnote{Also known as stable homotopy theories.}. Let $2\textrm{-}\mathrm{Ring}$ denote the $\infty$-category of  of essentially small, symmetric monoidal, stable $\infty$-categories with exact tensor product in each variable.  We refer to \cite[Definition 2.14]{Mat16}) for further details about these $\infty$-categories. 

\begin{example}\label{Example tt-ring of infinite degree in stmod}
    Let $G$ be the fundamental group of the following graph of finite groups,     
    
    
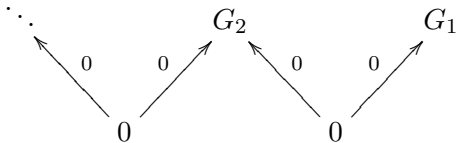
\centerline{\xymatrix{ 
\ddots \ar@{<-}[rd]^{0}  & &  G_2 \ar@{<-}[rd]^{0} \ar@{<-}[ld]_{0}  & & G_1 \ar@{<-}[ld]_0   \\ & 0   & &  0  &  
}
}

\noindent where $G_n$ is a non-trivial finite group, for $n\geq1$.  In other words, the group $G$ corresponds to the free product of the groups $G_n$. In particular, $G$ is a group of type $\Phi$ (see \cite{Tal}). By \cite[Theorem 3.3]{Gom23}, the stable module $\infty$-category $\StMod(kG)$ (see \cite[Definition 2.13]{Gom23}) decomposes in terms of the above graph of groups, that is, we have an equivalence
    $$\mathrm{StMod}(kG)\simeq \prod_{n\in\mathbb{N}} \mathrm{StMod}(kG_n)$$ in $\mathrm{CAlg}(\mathrm{Pr}^L_{\textrm{st}})$.  Note that dualizable objects in $\StMod(kG)$ are detected component-wise via this equivalence. In other words, we have a similar decomposition in $2\textrm{-}\mathrm{Ring}$ for the dualizable part of $\StMod(kG)$, i.e., the symmetric monoidal, stable $\infty$-category on the dualizable objects of $\StMod(kG)$.  Moreover, this factorization induces a product decomposition at the level of homotopy categories. Hence the homotopy category of the dualizable part of  $\StMod(kG)$ satisfies the hypothesis of Theorem \ref{example of infinite degree}. 
\end{example}


In practice, essentially small tensor triangulated categories arise as the dualizable part of a \textit{bigger} tensor triangulated category which, for instance, admits small coproducts, just as in Example \ref{Example tt-ring of infinite degree in stmod}. Then we can consider tt-rings in a tensor triangulated category which sits inside a bigger one. In particular, the framework of rigidly-compactly generated tensor triangulated categories has been extensively studied (see for instance \cite{BHS21}). In fact, all tt-rings that have been proved to have finite degree in \cite[Section 4]{Bal11} sit in the dualizable part a rigidly-compactly generated tensor triangulated category, so we might think these are the conditions we should impose on a tensor triangulated category to guarantee that any tt-ring has finite degree. We will see in Example \ref{compact tt-ring of infinite degree} that this is not the case.

 Recall that an object $x$ in a triangulated category $\K$ with small coproducts is \textit{compact} if the functor $\mathrm{Hom}(x,-)$ commutes with small coproducts. In particular, the subcategory $\K^c$ of compact objects remains triangulated.

\begin{definition}
    A tensor triangulated category $\K$ is \textit{rigidly-compactly generated} if $\K^c$ is essentially small, the smallest triangulated subcategory containing $\K^c$  which is closed under small coproducts is $\K$, and the class of compact objects coincides with the class of dualizable objects.  In this case, $\K^c$ remains tensor triangulated. 
\end{definition}

\begin{remark}\label{dualizables son compactos si la unidad lo es}

For a general tensor triangulated category $\K$ with small coproducts, compact objects are not necessarily dualizable, and vice versa, dualizable objects are not necessarily compact. However, if the monoidal unit of $\K$ is compact, then any dualizable object in $\K$ is compact. This follows from the fact that a dualizable object $x$ and its dual $y$ determine adjoint functors $x\otimes-\dashv y\otimes-$. 
\end{remark}

\begin{example}\label{compact tt-ring of infinite degree}
    For $i\in \mathbb{N}$, let $\K_i$ be a non-trivial rigid $2\textrm{-ring}$. Define $\K:=\prod_{i\in\mathbb{N}} \K_i$ in $2\textrm{-}\mathrm{Ring}$. Note that $\K$ is a rigid $2\textrm{-ring}$. Let $\mathcal{L}$ denote the $\mathrm{Ind}$-completion of $\K$ which lies in $\mathrm{CAlg}(\mathrm{Pr}^L_{\textrm{st}})$ (see \cite[Section 2]{NP23}). In particular, the compact objects of $\mathcal{L}$ are precisely the elements of $\K$. Since the inclusion functor $$\K\hookrightarrow \mathcal{L}$$ is strongly monoidal, we deduce that any compact element in $\mathcal{L}$ is dualizable. Therefore the homotopy category of $\mathcal{L}$ is a rigidly-compactly generated tensor triangulated category. In particular, we can construct a tt-ring in the dualizable part of $\mathcal{L}$, just as in Theorem \ref{example of infinite degree}, which has infinite degree.  
\end{example}

\noindent \begin{bf}Acknowledgments.\end{bf} I deeply thank my supervisor Jos\'e Cantarero for his support and for many interesting conversations on this work. I thank Paul Balmer and Luca Pol for helpful comments on this project. This work is part of the author's PhD thesis.

\bibliographystyle{alpha}
\bibliography{mybibfile}

\begin{thebibliography}{G{\'o}m23}

\bibitem[Bal10]{Bal10}
Paul Balmer.
\newblock Tensor triangular geometry.
\newblock In {\em Proceedings of the {I}nternational {C}ongress of
  {M}athematicians. {V}olume {II}}, pages 85--112. Hindustan Book Agency, New
  Delhi, 2010.

\bibitem[Bal14]{Bal11}
Paul Balmer.
\newblock Splitting tower and degree of tt-rings.
\newblock {\em Algebra Number Theory}, 8(3):767--779, 2014.

\bibitem[Bal16]{Bal16}
Paul Balmer.
\newblock Separable extensions in tensor-triangular geometry and generalized
  {Q}uillen stratification.
\newblock {\em Ann. Sci. \'{E}c. Norm. Sup\'{e}r. (4)}, 49(4):907--925, 2016.

\bibitem[BHS21]{BHS21}
Tobias Barthel, Drew Heard, and Beren Sanders.
\newblock Stratification in tensor triangular geometry with applications to
  spectral {M}ackey functors.
\newblock {\em arXiv preprint arXiv:2106.15540}, 2021.

\bibitem[G{\'o}m23]{Gom23}
Juan~Omar G{\'o}mez.
\newblock On the {P}icard group of the stable module category for infinite
  groups.
\newblock {\em arXiv preprint arXiv:2303.08260}, 2023.

\bibitem[Mat16]{Mat16}
Akhil Mathew.
\newblock The {G}alois group of a stable homotopy theory.
\newblock {\em Adv. Math.}, 291:403--541, 2016.

\bibitem[NP23]{NP23}
Niko Naumann and Luca Pol.
\newblock Separable commutative algebras and {G}alois theory in stable homotopy
  theories.
\newblock {\em arXiv preprint arXiv:2305.01259}, 2023.

\bibitem[Tal07]{Tal}
Olympia Talelli.
\newblock On groups of type {$\Phi$}.
\newblock {\em Arch. Math. (Basel)}, 89(1):24--32, 2007.

\end{thebibliography}

\end{document}